\documentclass[11pt]{amsart}
\usepackage[margin=2.8cm]{geometry}
\usepackage{amssymb,colordvi}
\usepackage{multirow}
\usepackage{bm}
\usepackage{float}
\usepackage{hyperref}
\usepackage{etoolbox}
\apptocmd{\thebibliography}{\setlength{\itemsep}{4pt}}{}{}
\usepackage{graphicx,color}
\numberwithin{equation}{section}
\newtheorem{theorem}{Theorem}
\numberwithin{theorem}{section}
\theoremstyle{example}
\newtheorem{example}{Example}[section]

\newtheorem{proposition}[theorem]{Proposition}

\newtheorem{lemma}[theorem]{Lemma}

\newtheorem{remark}[theorem]{Remark}

\DeclareMathOperator{\Vol}{Vol}

\title[Characterization of circuits with maximal number of positive solutions]{Characterization of circuits supporting polynomial systems with the maximal number of positive solutions}

\author{Boulos El Hilany}
\address{Laboratoire de Math\'ematiques\\
         Universit\'e Savoie Mont Blanc\\
         73376 Le Bourget-du-Lac Cedex\\
         France}

\email{boulos.el-hilany@univ-smb.fr}
\urladdr{https://lama.univ-savoie.fr/~elhilany/}
\usepackage{latexsym,stmaryrd}

\begin{document}
\maketitle

\begin{abstract}
A polynomial system with $n$ equations in $n$ variables supported on a set $\mathcal{W}\subset\mathbb{R}^n$ of $n+2$ points has at most $n+1$ non-degenerate positive solutions. Moreover, if this bound is reached, then $\mathcal{W}$ is minimally affinely dependent, in other words, it is a circuit in $\mathbb{R}^n$. For any positive integer number $n$, we determine all circuits $\mathcal{W}\subset\mathbb{R}^n$ which can support a polynomial system with $n+1$ non-degenerate positive solutions. Restrictions on such circuits $\mathcal{W}$ are obtained using Grothendieck's real dessins d'enfant, while polynomial systems with $n+1$ non-degenerate positive solutions are constructed using Viro's combinatorial patchworking.
\end{abstract}

\section{Introduction}

The support $\mathcal{W}$ of a polynomial $f\in\mathbb{C}[z_1^{\pm 1},\ldots,z_n^{\pm 1}]$ is the set of points $w\in\mathbb{Z}^n$ corresponding to monomials $z^w$ appearing with a non-zero coefficient in $f$. The support of a polynomial system \begin{equation}\label{sys:Gen}
f_1(z)=\cdots=f_n(z)=0,\tag{$\star$}
\end{equation} in $n$ variables is the union of supports of the individual polynomials $f_1,\ldots,f_n$. An important result by Kouchnirenko in ~\cite{Kou} shows that the number of non-degenerate solutions of ~\eqref{sys:Gen} in $(\mathbb{C}^*)^n$ is at most $n!\Vol(\Delta_\mathcal{W})$, where $\Vol(\Delta_\mathcal{W})$ is the standard Euclidean volume of the convex hull $\Delta_\mathcal{W}$ of $\mathcal{W}$. Real polynomial systems (defined by polynomials with real coefficients) appear ubiquitously in mathematics, and we are interested in estimating their numbers of real solutions. Although both Bezout's and Kouchnirenko' bounds hold true for the number of non-degenerate solutions in $(\mathbb{R^*})^n$ as well, the resulting bounds are not always sharp. This typically happens when $\mathcal{W}$ has few elements comparatively to $\Delta_\mathcal{W}\cap\mathbb{Z}^n$. A particular attention is paid on the positive solutions of ~\eqref{sys:Gen}, which are the solutions contained in the positive orthant of $\mathbb{R}^n$. Indeed, assume that there exists a sharp upper bound $N_{\mathcal{W}}$ on the number of non-degenerate positive solutions of ~\eqref{sys:Gen} that depends only on $\mathcal{W}$. Then this $N_{\mathcal{W}}$ also bounds the number of solutions contained in any other orthant, and thus ~\eqref{sys:Gen} will not have more than $2^nN_{\mathcal{W}}$ solutions in $(\mathbb{R}^*)^n$. Descartes showed in 1637 that we have $N_{\mathcal{W}}=|\mathcal{W}| -1$ for $n=1$, but still, before Khovanskii's book ~\cite{Kho}, it was not clear that such $N_{\mathcal{W}}$ even exists for $n\geq 2$. He proved that the number of non-degenerate positive solutions of a system of $n$ equations in $n$ variables having a total of $n+k+1$ distinct monomials is bounded by $2^{n+k \choose 2}(n+1)^{n+k}$. This bound was later improved by F. Bihan and F. Sottile ~\cite{BS} to $\frac{e^2 + 3}{4}2^{k \choose 2}n^k$, however only a handful of sharp fewnomial bounds are known (c.f. ~\cite{B,BE,LRW}). When $k=0$, the system ~\eqref{sys:Gen} can be reduced to a system where each polynomial is a binomial, and thus has at most one non-degenerate positive solution. The previous bounds on the number of non-degenerate positive solutions also hold true for (generalized) polynomial systems with support $\mathcal{W}\subset\mathbb{R}^n$, which are systems where each equation is a linear combination of monomials with real exponents.

 One of the first non-trivial cases arises when $n\geq 2$ and $k=1$, in which case the support is a set of $n+2$ points in $\mathbb{R}^n$. F. Bihan ~\cite{B} proved that any polynomial system supported on such $\mathcal{W}$ has at most $n+1$ non-degenerate positive solutions. Moreover, if this bound is reached, then $\mathcal{W}$ is minimally affinely dependent, which means that it is a \textit{circuit} in $\mathbb{R}^n$. In the following we assume that $\mathcal{W}$ is a circuit in $\mathbb{R}^n$. Note that up to a scalar multiplication, there exists a unique affine relation on $\mathcal{W}$. Polynomial systems supported on a circuit in $\mathbb{Z}^n$ whose all non-degenerate complex solutions are positive have been studied in ~\cite{B1} (such systems are called \textit{maximally positive}). As a main result, it is given for any positive integer $n$ a finite list of circuits in $\mathbb{Z}^n$ that can support maximally positive systems up to the obvious action of the group of invertible integer affine transformations of $\mathbb{Z}^n$. Also for the circuit case, a very recent generalization of the Descartes' rule of sign was developed by F. Bihan and A. Dickenstein in ~\cite{BD}. This gave some conditions on both the circuit and the coefficient matrix that are necessary for the system to have $n+1$ non-degenerate positive solutions. More precisely, the authors in ~\cite{BD} show that if such system has $n+1$ non-degenerate positive solutions, then all maximal minors of the coefficient matrix are nonzero and any affine relation $\sum_{i=1}^{n+2} \lambda_iw_i=0$ on $\mathcal{W}$ has the same number (up to 1 if $n$ is odd) of positive coefficients as that of negative ones. \

In this paper, we completely characterize the circuits which are supports of polynomial systems with $n+1$ non-degenerate positive solutions.

\begin{theorem}\label{Main:Th.0}
A circuit $\mathcal{W}$ in $\mathbb{R}^n$ supports a system with $n+1$ non-degenerate positive solutions if and only if there exists a bijection 
$$ 
  \begin{array}{ccc}
   \lbrace 1,\ldots,n+2 \rbrace & \longrightarrow & \mathcal{W} \\
    
    i & \longmapsto  & w_i
  \end{array}
$$ such that every affine relation on $\mathcal{W}$ can be written as $$\sum_{i=1}^{s}\alpha_iw_i = \sum_{s+1}^{n+2}\beta_iw_i,$$ where $s=\lfloor (n+2)/2 \rfloor$ and all $\alpha_i$, $\beta_i$ are positive numbers which satisfy $$\sum_{i=1}^{r}\alpha_i< \sum_{i= s + 1 }^{s + r }\beta_i< \sum_{i=1}^{r+1}\alpha_i \text{ \quad for\quad} r=1,\ldots, s-1 \text{\quad if\quad} n \text{\quad is even}$$ or $$\sum_{i=1}^{r}\alpha_i< \sum_{i=s+2  }^{s + r +1 }\beta_i< \sum_{i=1}^{r+1}\alpha_i \text{ \quad for\quad} r=1,\ldots, s-1 \text{\quad if\quad} n \text{\quad is odd.}$$

\end{theorem}

%
%

F. Bihan proved in ~\cite{B1} that if a circuit in $\mathbb{Z}^n$ supports a maximally positive system with $n+1$ non-degenerate positive solutions, then it has a primitive affine relation (i.e. affine relation with coprime integer coefficients) as in Theorem ~\ref{Main:Th.0} with $\alpha_1=\beta_{n+2}=1$ and all other coefficients are equal to two. This can be seen as a consequence of Theorem ~\ref{Main:Th.0} (see Example ~\ref{Ex:1}). Indeed, if $\mathcal{W}$ supports a maximally positive system with $n+1$ non-degenerate positive solutions, then the subgroup of $\mathbb{Z}^n$ generated by $\mathcal{W}$ is $\mathbb{Z}^n$. Moreover, if $\sum_{i=1}^{s}\alpha_iw_i = \sum_{s+1}^{n+2}\beta_iw_i$ is a primitive affine relation, then $\sum_{i=1}^{s}\alpha_i = \sum_{s+1}^{n+2}\beta_i=n+1$ (see ~\cite{B1} for more details), which together with inequalities in Theorem ~\ref{Main:Th.0} imply the desired equalities.\

In order to prove Theorem ~\ref{Main:Th.0}, one can reduce to the case where $\mathcal{W}\subset\mathbb{Z}^n$. Indeed, assume that a system with support $\mathcal{W}=\{w_1,\ldots,w_{n+2}\}\subset\mathbb{R}^n$ has $n+1$ non-degenerate positive solutions. Then for $i=1,\ldots,n+2$, points $\tilde{w}_i\in\mathbb{Q}^n$ that are sufficiently close to $w_i$ support a (generalized) polynomial system with the same coefficients and having at least $n+1$ non-degenerate positive solutions, and thus exactly this number of non-degenerate positive solutions since $n+1$ is an upper bound. Now, multiplying all $\tilde{w}_i$ by some integer, one acquires a system supported on a circuit in $\mathbb{Z}^n$ with $n+1$ non-degenerate positive solutions. Since the inequalities appearing Theorem ~\ref{Main:Th.0} are strict, if the first circuit $\mathcal{W}$ satisfies them, then they are satisfied by the new circuit $\tilde{\mathcal{W}}$ as well, and vice-versa.

 Assume that $\mathcal{W}=\{w_1,\ldots,w_{n+2}\}$ is a set of $n+2$ points in $\mathbb{Z}^n$  and consider any affine relation $\sum_{i=1}^{n+2} \lambda_iw_i=0$ with integer coefficients. After a small perturbation, any system with $n$ equations in $n$ variables $z=(z_1,\ldots,z_n)$ and supported on $\mathcal{W}$ can be reduced by Gaussian elimination to a system \begin{equation}\label{eq:sys:circ}
z^{w_i}=P_i(z^{w_{n+1}})\text{\quad for \quad} i=1,\ldots,n, 
 \end{equation} having at least the same number of non-degenerate positive solutions, where $P_1,\ldots,P_{n+1}$ are real polynomials of degree 1 in one variable (see Section ~\ref{Sec1}). We define in Section ~\ref{Sec1} a real rational function $\varphi(y)=\prod_{i=1}^{n+1}P_i^{\lambda_i}$, where ~\eqref{eq:sys:circ} can be completely recoverd from the function $\varphi$ and vice-versa. We apply \textit{Gale duality} (c.f. ~\cite{B1,BS1,BS}) to obtain a correspondence between non-degenerate solutions of a system supported on $\mathcal{W}$ and those of $\varphi(y) = 1$. This correspondence restricts to a bijection between non-degenerate positive solutions of the system and the solutions contained in the (possibly empty) interval $\Delta_+:=\{y\in\mathbb{R}|\ P_i(y)>0 \text{\ for\ }i=1,\ldots,n+1\}$. After homogenization, we get a real rational map $\mathbb{C}P^1\rightarrow\mathbb{C}P^1$ that we denote again by $\varphi$. The \textit{real dessin d'enfant} $\Gamma$ associated to $\varphi: \mathbb{C}P^1\rightarrow\mathbb{C}P^1$, is the inverse image of the real projective line under $\varphi$. Given that $\varphi(y)=1$ has $n+1$ solutions in $\Delta_+$, we deduce by analyzing $\Gamma$ in Section ~\ref{Sec2} the inequalities of Theorem ~\ref{Main:Th.0}.
 
  The solutions of $\varphi(y)=1$ in $\Delta_+$ are roots of the \textit{Gale polynomial} \begin{equation}\label{eq:Gale0}
 G(y)=\prod_{\lambda_i>0}P_i^{\lambda_i}(y) - \prod_{\lambda_i<0}P_i^{-\lambda_i}(y)
\end{equation} in the same interval. In ~\cite{PR} (see the proof of Lemma 1.8), K. Phillipson and J.-M. Rojas construct polynomial systems ~\eqref{eq:sys:circ} supported on a circuit in $\mathbb{Z}^n$ with $n+1$ non-degenerate positive solutions using \textit{Viro polynomials} $P_{i,t}(y)=a_i+t^{\alpha_i}b_i$, where $a_i,b_i, \alpha_i\in\mathbb{R}$, and $t>0$ is a parameter that will be taken small enough. They apply the version of Viro's \textit{combinatorial patchworking} developed in ~\cite{St} which involves mixed subdivision of Newton polytopes. Here, we also use Viro polynomials $P_{i,t}$, but look directly at the roots of the corresponding Gale polynomial in $\Delta_+$. The inequalities in Theorem ~\ref{Main:Th.0} appear explicitly as being necessary to construct polynomial systems ~\eqref{eq:sys:circ} with $n+1$ non-degenerate positive solution using Viro polynomials $P_{i,t}$.\\

\textbf{Acknowledgments:} The author is grateful to Fr\'ed\'eric Bihan for very helpful discussions and suggestions.

\section{Technical Preamble}\label{Sec1}

Given a system of $n$ polynomials in $n$ variables with total support a circuit $\mathcal{W}=\{w_1,\ldots,w_{n+2}\}$, perturbing slightly its coefficients if necessary, we may assume that the coefficients of $z^{w_1} ,\ldots, z^{w_n}$ in the system form an invertible matrix (a small perturbation does not decrease the number of non-degenerate positive solutions). Since we are only interested in non-degenerate positive solutions, we may assume that $w_{n+2}=\bm{0}$ and we transform ~\eqref{sys:Gen} via Gaussian elimination  into an equivalent system such that the coefficients of $z^{w_1} ,\ldots, z^{w_n}$ form a diagonal matrix

\begin{equation}\label{sys:main}
z^{w_i}=P_i(z^{w_{n+1}})\text{\quad for \quad} i=1,\ldots,n,
\end{equation} where $P_i(z^{w_{n+1}})=a_i + b_iz^{w_{n+1}}$ for $i=1,\ldots,n$. We start by giving a brief description about \textit{Gale duality} for the system ~\eqref{sys:main} (c.f. ~\cite{B1,B,BS1}). We use the linear relations on $\mathcal{W}$ to obtain a special polynomial in one variable, called \textit{Gale polynomial}. We have that any integer linear relation among the exponent vectors of $\mathcal{W}$
\begin{equation}\label{eq:affinerel}
\sum_{i=1}^{n+1}\lambda_iw_i=0
\end{equation} gives a monomial identity $$ (z^{w_1})^{\lambda_1}\cdots (z^{w_n})^{\lambda_n}(z^{w_{n+1}})^{\lambda_{n+1}}=1.$$ If we substitute the polynomials $P_i(z^{w_{n+1}})$ of ~\eqref{sys:main} into this identity, we obtain a consequence of the latter equation
 
\begin{equation} \label{Gale.sys.2}
 (P_1(z^{w_{n+1}}))^{\lambda_1}\cdots(P_n(z^{w_{n+1}}))^{\lambda_n}(z^{w_{n+1}})^{\lambda_{n+1}}=1.
\end{equation} Under the substitution  $y = z^{w_{n+1}}$, the polynomials $P_i(z^{w_{n+1}})$ become linear functions $P_i(y)$. Set $P_{n+1}(y)= y$. Then ~\eqref{Gale.sys.2} becomes

\begin{equation} \label{Gale.sys.3}
 \displaystyle \overset{n+1}{\underset{i=1}{\prod}} P_i(y)^{\lambda_i} = 1,
\end{equation} which constitutes a \textbf{Gale transform} associated to the system ~\eqref{sys:main}. Recall that $$\Delta_+=\{y\ |\ P_i(y)>0\ \text{for}\ i=1,\ldots , n+1\}.$$ We can write equivalently ~\eqref{Gale.sys.3} as $G(y)=0$, where $G$ is the \textbf{Gale polynomial} defined by

\begin{equation}\label{Gale4}
G(y)= \prod_{\lambda_i>0} P_i^{\lambda_i}(y) - \prod_{\lambda_i<0} P_i^{-\lambda_i}(y).
\end{equation}

\begin{proposition}{~\cite{BS}} \label{Th.bij.Gale}
The association
$$\phi_{w_{n+1}}\ :\ \mathbb{R}_+^n \ni z\ \longmapsto z^{w_{n+1}} =: y\in\mathbb{R}_+$$
is a bijection between solutions $z\in\mathbb{R}_+^n$ of the diagonal system ~\eqref{sys:main} and solutions $y\in\Delta_+$ of ~\eqref{Gale.sys.3} which restricts to a bijection between their non-degenerate solutions.
\end{proposition}
\section{Proof of the "only if" direction of Theorem ~\ref{Main:Th.0}}\label{Sec2}

Set $P_{n+2}(y)=1$ and $\lambda_{n+2}= -\sum_{i=1}^{n+1}\lambda_i$. We see $P_{n+2}$ as a polynomial of degree 1 having a root at $\infty$. In what follows, we study the solutions of $\varphi (y)=1$ contained in $\Delta_+$ where 

\begin{equation} \label{eq:Gale.main}
\varphi(y)=\prod_{i=1}^{n+2}P^{\lambda_i}_i(y).
\end{equation} 

 We say that a point $x\in\mathbb{R}\cup\{\infty\}$ is a \textbf{special point} of $\varphi$ if $x$ is either a root or a pole of $\varphi$. Conversely, a \textbf{non-special critical point} $x\in\mathbb{R}$ of $\varphi$ is a root of $\varphi'$ such that $x$ is not a special point of $\varphi$. In what follows, we see $\varphi$ (after homogenization) as a real rational map $\mathbb{C}P^1\rightarrow\mathbb{C}P^1$.

\begin{remark}\label{rem:non-spec}
It is proven in ~\cite{B} (Proof of Proposition 2.1.) that

\begin{equation}\label{derivative}
 \varphi'(y)= y^{\lambda_{n+1}-1}\prod_{i=1}^{n}P_i^{\lambda_i -1}(y)\cdot H(y),
\end{equation} where $\deg H\leq n$. Therefore $\varphi$ has at most $n$ non-special critical points.

\end{remark} Assume that $\Delta_+$ is a non-empty interval. Note that all special points of $\varphi$ are contained in $\mathbb{R}P^1$, and that by definition, the endpoints of $\Delta_+$ are special points of $\varphi$. Choose an orientation of $\mathbb{R}P^1$ and enumerate the special points $x_1,\ldots,x_{n+2}$ of $\varphi$ with respect to this orientation so that $x_{i}<x_{i+1}$ for $i=1,\ldots,n+1$ and the endpoints of $\Delta_+$ are $x_1$ and $x_{n+2}$ (see Figure ~\ref{Delta}). We also renumber the polynomials $P_i$ so that $x_i$ is the root of $P_i$ for $i=1,\ldots,n+2$.

\begin{figure}[H]
\centering
\includegraphics[scale=0.4]{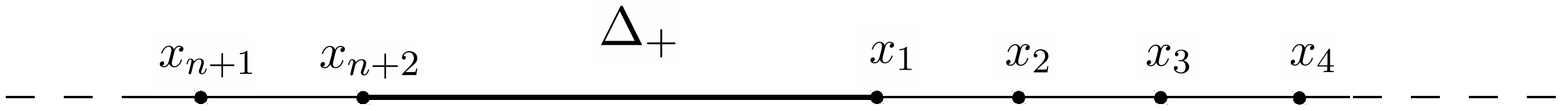}
\caption{\ }
\label{Delta}
\end{figure} The real dessin d'enfant associated to $\varphi$ is $\Gamma:=\varphi^{-1}(\mathbb{R}P^1)$. The notations we use are taken from ~\cite{B} and for a more detailed description of real dessins d'enfant see ~\cite{Br,O} for instance. The dessin d'enfant $\Gamma$ is a graph on $\mathbb{C}P^1$ invariant with respect to the complex conjugation and which contains $\mathbb{R}P^1$. Any connected component of $\mathbb{C}P^1\setminus\Gamma$ is homeomorphic to an open disk. Real critical points of $\varphi$ with real critical value are vertices of $\Gamma$. Moreover, each vertex of $\Gamma$ has even valency, and the multiplicity of a critical point with real critical value of $\varphi$ is half its valency. The graph $\Gamma$ contains the inverse images of $0$, $\infty$ and $1$. The inverse image of $0$ (resp. $\infty$) are the roots of all $P_i$ where $\lambda_i$ are positive (resp. negative). Denote by the same letter $p$ (resp. $q$ and $r$) the points of $\Gamma$ which are mapped to $0$ (resp. $\infty$ and $1$). Orient the real axis on the target space via the arrows $0\rightarrow \infty\rightarrow 1\rightarrow 0$ (orientation given by the decreasing order in $\mathbb{R}$) and pull back this orientation by $\varphi$. The graph $\Gamma$ becomes an oriented graph, with the orientation given by arrows $p\rightarrow q\rightarrow r\rightarrow p$. The graph $\Gamma$ is called {\it real dessin d'enfant} associated to $\varphi$. A \textbf{cycle} of $\Gamma$ is the boundary of a connected component of $\mathbb{C}P^1\backslash\Gamma$. Any such cycle contains the same non-zero number of letters $r$, $p$ , $q$ ordered by $p\rightarrow q\rightarrow r\rightarrow p$ (see Figure ~\ref{fleches}): we say that a cycle obeys the \textbf{cycle rule}. \begin{figure}[H]
\centering
\includegraphics[scale=1.5]{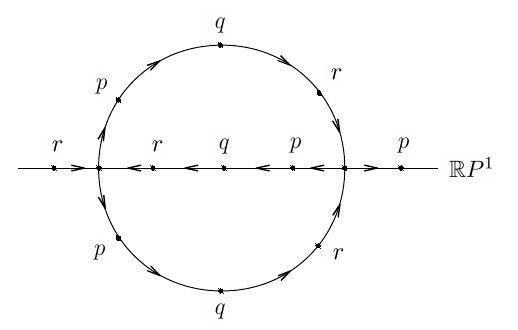}
\caption{Cycles of $\Gamma$ obeying the cycle rule.}
\label{fleches}
\end{figure}

 Since the graph is invariant under complex conjugation, it is determined by its intersection with one connected component $H$ (for half) of $\mathbb{C}P^1\setminus\mathbb{R}P^1$. In most figures we will only show one half part $H\cap\Gamma$ together with $\mathbb{R}P^1=\partial H$ represented as a horizontal line. Moreover, for simplicity, we omit the arrows. Let $a$, $b$ be two critical points of $\varphi$ i.e. vertices of $\Gamma$. We say that $a$ and $b$ are \textbf{neighbors} if there is a branch of $\Gamma\setminus\mathbb{R}P^1$ joining them such that this branch does not contain any special or critical points of $\varphi$ other than $a$ or $b$. In what follows, we assume that $\varphi(y)=1$ has $n+1$ solutions contained in $\Delta_+$. Since the latter interval does not contain special points of $\varphi$, by Rolle's theorem, the function $\varphi$ has at least $n$ non-special critical points in $\Delta_+$, and by Remark ~\ref{rem:non-spec}, the non-special critical points of $\varphi$ (all $n$ of them) are contained in $\Delta_+$.

\begin{lemma}\label{lambda}
We have $\lambda_i\lambda_{i+1}<0$ for $i=1,\ldots,n+1$.
\end{lemma}

\begin{proof}
 Consider a couple $x_i$, $x_{i+1}$ of two consecutive special points of $\varphi$ with $i\in\{1,\ldots,n+1\}$. Then these two points are endpoints of an open interval in $\mathbb{R}P^1$ which does not contain special points or non-special critical points. By the cycle rule, this implies that one endpoint is a root (letter $p$) and the other is a pole (letter $q$) of $\varphi$. \end{proof}

We will assume that for $i=1,\ldots,n+2$, we have $\lambda_{i}>0$ if $i$ is odd, and $\lambda_{i}<0$ if $i$ is even.

\begin{lemma}\label{L:c_and_c_no_neighbors}
 The non-special critical points of $\varphi$ cannot be neighbors to each other.
\end{lemma}

\begin{proof}

First, note that all special points of $\varphi$ are contained in $\mathbb{R}P^1\setminus\Delta_+$. Consider the branch of $\Gamma$ contained in one of the connected components of $\mathbb{C}P^1\setminus\mathbb{R}P^1$ joining two non-special critical points. Then one of the two connected components of $\mathbb{C}P^1\setminus\Gamma$ adjacent to this edge will have a boundary disobeying the cycle rule.\end{proof}

\begin{lemma}\label{L:p_and_2c_no_neigh}
 A special critical point of $\varphi$ cannot be a neighbor to more than one non-special critical point. 
 \end{lemma}

\begin{proof}
 Assume that there exists a special critical point $\alpha$ of $\varphi$ that is a neighbor to at least two non-special critical points of $\varphi$ (in $\mathbb{R}P^1$). Let $c_1$ and $c_2$ be two such consecutive non-special critical points. Consider two branches of $\Gamma$ contained in one of the connected components of $\mathbb{C}P^1\setminus\mathbb{R}P^1$ joining $\alpha$ to $c_1$ and $\alpha$ to $c_2$ respectively. Then one of the two connected components of $\mathbb{C}P^1\setminus\Gamma$ adjacent to these two branches will have a boundary containing only $\alpha$ as a special point, and thus disobeying the cycle rule.
\end{proof}

\begin{lemma}\label{L:c_and_v_no_neighbors}
The special points $x_1$ and $x_{n+2}$ of $\varphi$ are not neighbors to any of the non-special critical points.
\end{lemma}

\begin{proof}
Assume that $x_1$ is a neighbor to a non-special critical point $c$ (the case where $x_{n+2}$ is a neighbor to $c$ is symmetric). Recall that $\Delta_+$ does not contain special points of $\varphi$. Consider the branch of $\Gamma$ contained in one of the connected components of $\mathbb{C}P^1\setminus\mathbb{R}P^1$ joining $x_1$ to $c$. Then one of the two connected components of $\mathbb{C}P^1\setminus\Gamma$ adjacent to this branch will have a boundary containing only $x_1$ as a special point, and thus disobeying the cycle rule.\end{proof}

 Recall that $\varphi$ has $n$ non-special critical points  all contained in $\Delta_+$. Let $c_2,\ldots,c_{n+1}$ denote these points numbered so that $x_{n+2}<c_{n+1}<c_n<\cdots <c_2<x_1$.
 
\begin{proposition} \label{prop:Main1}
For $i=2,\ldots,n+1$, the special point $x_i$ is a neighbor to $c_i$ (see Figure ~\ref{Fig1}).
\end{proposition}
 
\begin{proof}
 First, by Lemma ~\ref{L:c_and_v_no_neighbors}, we have that the roots of $P_1$ and $P_{n+2}$ are not neighbors to non-special critical points. Recall that there exists $n$ non-special critical points in $\Delta_+$. Therefore, by Lemmata ~\ref{L:c_and_c_no_neighbors} and ~\ref{L:p_and_2c_no_neigh}, we have that for $i = 2,\cdots, n + 1$, the special point $x_i$ is a neighbor to only one non-special critical point $c_j$. Consider the closed interval $I\subset\mathbb{R}P^1$ with endpoints $x_i$ and $c_j$ and which contains $x_1$. The special points in $I$ are $x_1,x_2,\ldots,x_i$ and the non-special critical points in $I$ are $c_2,\ldots,c_j$. Then the non-special critical points in $I$ can only be neighbors to special points in $I\setminus\{x_1\}$ (see Lemma ~\ref{L:c_and_v_no_neighbors}). This induces a bijection between $\{x_2,\ldots,x_i\}$ and $\{c_2,\ldots,c_j\}$, thus $i=j$.

\begin{figure}[H]
\centering
\includegraphics[scale=1.5]{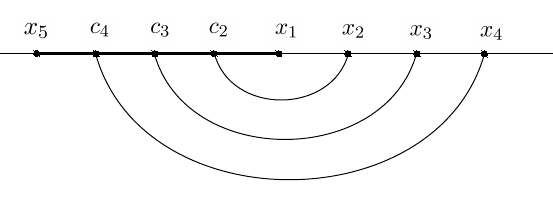}
\caption{\ }
\label{Fig1}
\end{figure}\end{proof}

\begin{lemma}\label{L:key:x_1x_(n+2)}

The special point $x_1$ (resp. $x_{n+2}$) of $\varphi$ can only be a neighbor to the special point $x_2$ (resp. $x_{n+1}$) of $\varphi$.

\end{lemma}

\begin{proof}
We prove the result only for $x_1$ since the case for $x_{n+2}$ is symmetric. Consider the open interval $I$ with endpoints $c_2$ and $x_2$ containing $x_1$. By Proposition ~\ref{prop:Main1}, we have that $c_2$ and $x_2$ are neighbors. The result comes as a consequence of Lemma ~\ref{L:c_and_v_no_neighbors} and of the fact that there does not exist special points or non-special critical points in $I$ other than $x_1$ (See Figure ~\ref{Fig3}).
\end{proof}

\begin{lemma}\label{L:key:x_i}
For $i=1,\ldots, n$, the only special points which can be neighbors to $x_{i+1}$ are $x_i$ and $x_{i+2}$.

\end{lemma}

\begin{proof}
 Assume first that $i=1$ (the case $i=n$ is symmetric). Recall that by Proposition ~\ref{prop:Main1}, the special point $x_2$ (resp. $x_3$) and $c_2$ (resp. $c_3$) are neighbors. Therefore, the only other possible neighbors to $x_2$ are $x_1$ and $x_3$ (see Figure ~\ref{Fig3}).\\

\begin{figure}[H]
\centering
\includegraphics[scale=1.5]{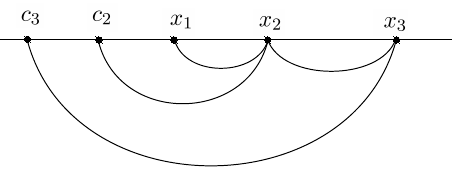}
\caption{\ }
\label{Fig3}
\end{figure}

 Assume now that $i\neq 1$ and $i\neq n$. Recall that by Proposition ~\ref{prop:Main1} the point $x_{i}$ (resp. $x_{i+2}$) is a neighbor to $c_i$ (resp. $c_{i+2}$). Consider the disc $\mathcal{C}$ in $\mathbb{C}P^1$ with boundary given by the union of $[c_{i+2},c_{i}]$, $[x_i,x_{i+2}]$ and the complex arcs of $\Gamma$ joining $c_i$ to $x_i$ (resp. $c_{i+2}$ to $x_{i+2}$), and which are contained in one given connected component of $\mathbb{C}P^1\setminus\mathbb{R}P^1$ (see Figure ~\ref{Fig4}). The result follows from the fact that the only special points in $\mathcal{C}$ are $x_i$ and $x_{i+2}$.

\begin{figure}[H]
\centering
\includegraphics[scale=1.5]{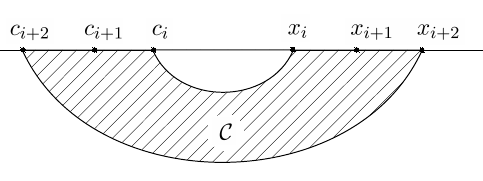}
\caption{\ }
\label{Fig4}
\end{figure}\end{proof}

Recall that $\lambda_i$ is positive if $i$ is odd and negative if $i$ is even, and thus the root $x_i$ of $P_i$ is a zero (resp. pole) of $\varphi$ if $i$ is odd (resp. even). Recall that the valency of any special point $x_i$ is the number $V_i$ of edges of $\Gamma$ that are incident to $x_i$.

For $i=1,\ldots,n+1$, denote by $N_{i,i+1}$ the number of edges of $\Gamma$ in $\mathbb{C}P^1\setminus\mathbb{R}P^1$ joining the special points $x_i$ and $x_{i+1}$. By Lemmata ~\ref{L:c_and_v_no_neighbors} and ~\ref{L:key:x_1x_(n+2)}, we have $V_1=N_{1,2}+2$ and $V_{n+2}=N_{n+1,n+2}+2$ (each number $2$ corresponds to the pair of edges of $\Gamma$ in $\mathbb{R}P^1$ incident to $x_1$ and $x_{n+2}$ respectively). Moreover, for $i=2,\ldots,n+1$, Proposition ~\ref{prop:Main1} and Lemma ~\ref{L:key:x_i} show that $V_i = N_{i-1,i} +N_{i,i+1} +4$, where the number $4$ counts the branches in $\mathbb{R}P^1$ together with the branches joining $x_i$ to $c_i$. Knowing that $V_i=|2\lambda_i|$, it is straightforward to compute that for $k=1,\ldots,\lfloor n/2\rfloor +1$, we have

\begin{equation}\label{ifneven1}
  \sum_{j=1}^k\lambda_{2j-1} < -\sum_{j=1}^k\lambda_{2j}<\sum_{j=0}^k\lambda_{2j+1} \text{\quad if $n$ is even, or}
\end{equation}
\begin{equation}\label{ifnodd}
\sum_{j=1}^k\lambda_{2j-1} < -\sum_{j=1}^k\lambda_{2j}<\sum_{j=0}^k\lambda_{2j+1} \text{\quad if $n$ is odd.}
\end{equation} This finishes the proof of the "only if part" of Theorem ~\ref{Main:Th.0}.

 We now finish the description of $\Gamma$. For $i\in\{0,\ldots,n+1\}$, consider the real branch $L_0$ joining two consecutive special points $x_i$ and $x_{i+1}$ of $\varphi$. Let $k:=N_{i,i+1}/2$, and for $j=1,\ldots, k$, consider the couple of conjugate branches $(L_j,\overline{L}_j)$ joining $x_i$ to $x_{i+1}$ enumerated such that the open disc of $\mathbb{C}P^1$ with boundary $(L_j,\overline{L}_j)$ and containing $L_0$, contains the couple $(L_{j-1},\overline{L}_{j-1})$ as well (assuming that $L_0\equiv\overline{L}_0$). The branch $L_k$ (resp. $\overline{L}_k$) does not contain a letter $r$ since there exists a cycle of $\Gamma_1$ containing both $L_k$ (resp. $\overline{L}_k$) and a letter $r\in\Delta_+$, and thus obeying the cycle rule. On the other hand, the branch $L_{k -1}$ (resp. $\overline{L}_{k-1}$) contains a letter $r$ where the cycle formed by the union of $L_{k}$ and $L_{k-1}$ (resp. $\overline{L}_{k}$ and $\overline{L}_{k-1}$) and containing $x_i$ and $x_{i+1}$ obeys the cycle rule. We deduce that for $j=0,\ldots,k$, the branch $L_j$ (resp. $\overline{L}_j$) has exactly 1 or 0 letters $r$ according as $j$ and $k-1$ have the same parity or not (see Example ~\ref{Ex:2}).

In fact, this complete description of the dessin d'enfant $\Gamma$ can be used to prove the "if" part of Theorem ~\ref{Main:Th.0} with the same techniques as in ~\cite{B}. However, we choose in Section ~\ref{Sec3} a different method, namely Viro's combinatorial patchworking, which shows clearly why the inequalities of Theorem ~\ref{Main:Th.0} are necessary. 
\begin{remark}
 From the relations described above, we see that the collection of integers $N_{i,i+1}$ is determined by the collection of the coefficients $\lambda_i$ (and vice-versa). Moreover, we see that the inequalities of Theorem ~\ref{Main:Th.0} are equivalent to $N_{i,i+1}\geq 0$ for $i=1,\ldots,n+1$.
\end{remark}

\begin{figure}[H]
\centering
\includegraphics[scale=1.5]{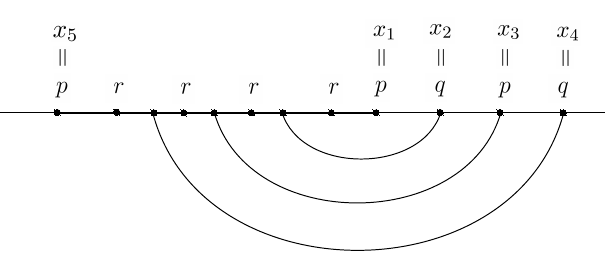}
\caption{\ }
\label{Fig:Gamma3}
\end{figure} 

\begin{example}\label{Ex:2}
Figure ~\ref{Fig:Gamma3+} represents an example of $\Gamma$ where $n=3$, $\lambda_1=3$, $\lambda_2=-7$, $\lambda_3=6$, $\lambda_4=-3$ and $\lambda_5=1$. The dessin d'enfant $\Gamma$ can be obtained from $\Gamma_0$ (see Figure ~\ref{Fig:Gamma3}) by adding complex branches connecting consecutive special points and letters $r$ as described above. 

\end{example}

\begin{figure}[H]
\centering
\includegraphics[scale=1.5]{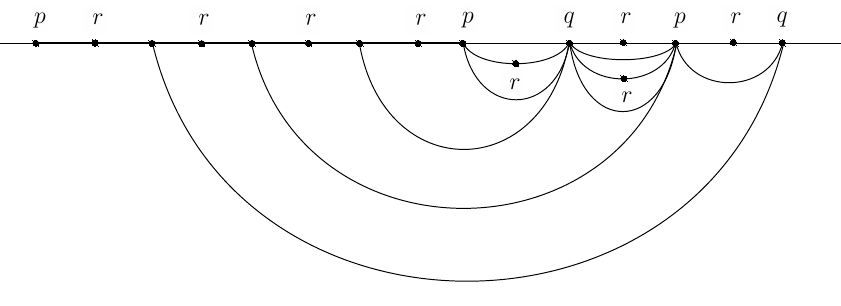}
\caption{\ }
\label{Fig:Gamma3+}
\end{figure} 

\section{Proof of the "if" direction of Theorem ~\ref{Main:Th.0}}\label{Sec3}

Assume that $\lambda_i>0$ if $i$ is odd, $\lambda_i<0$ if $i$ is even and ~\eqref{ifneven1} or ~\eqref{ifnodd} is satisfied (depending on the parity of $n$). In this section, we construct polynomials $P_i$ (see Section ~\ref{Sec2}) such that ~\eqref{eq:Gale.main} has $n+1$ solutions in $\Delta_+$. These polynomials have the form $P_{1,t}(y)=t^{\alpha_1}y$, $P_{n+2,t}(y)=1$ and $P_{i,t}(y)=1 + t^{\alpha_i}y$ for $i=2,\ldots, n+1$, where $t$ is a real positive parameter that will be taken small enough, and each $\alpha_i$ is a real number. The corresponding Gale polynomial ~\eqref{Gale4} is 

\begin{equation}\label{eq:Gale2}
\displaystyle G_t(y):=\prod_{j=0}^{\lfloor n/2\rfloor} P_{2j+1,t}^{\lambda_{2j+1}}(y) - \prod_{j=1}^{\lfloor (n+1)/2\rfloor} P_{2j,t}^{-\lambda_{2j}}(y).
\end{equation}

\noindent We are interested in the roots of $G_t$ contained in $\Delta_{+,t}$, which is the common positivity domain of the polynomials $P_{i,t}$. Note that here $\Delta_{+,t}=]0,+\infty[$. The polynomial $G_t$ is a particular case of a Viro polynomial (c.f. ~\cite{BBS,Ol1,Ol2}) $$f_t(y)=\sum_{p=p_0}^d\phi_p(t)y^p,$$ where $t$ is a positive real number, and each coefficient $\phi_p(t)$ is a finite sum $\sum_{q\in I_p} c_{p,q} t^q$ with $c_{p,q}\in\mathbb{R}$ and $q$ a real number.

We now recall how one can recover in some cases the real roots of $f_t$ for $t$ small enough (see for instance ~\cite{BBS}). Write $f$ for the function of $y$ and $t$ defined by $f_t$. Let $D\subset\mathbb{R}^2$ be the convex hull of the points $(p,q)$ for $p_0\leq p\leq d$ and $q\in I_p$. Assume that $D$ has dimension 2. Its lower hull $L$ is the union of the edges $e_1,\ldots,e_l$ of $D$ whose inner normals have positive second coordinate. Let $I_i$ be the image of $e_i$ under the projection $\mathbb{R}^2\rightarrow\mathbb{R}$ forgetting the last coordinate. Then the intervals $I_1,\ldots,I_l$ subdivide the Newton segment $[p_0, d]$ of $f_t$. Let $f^{(i)}$ be the facial subpolynomial of $f$ for the face $e_i$. That is, the polynomial $f^{(i)}$ is the sum of terms $c_{p,q}y^p$ such that $(p,q)\in e_i$. Suppose that $e_i$ is the graph of $y\mapsto a_i y+b_i$ over $I_i$. Expanding $f_t(yt^{- a_i})/t^{b_i}$ in powers of $t$ gives 

\begin{equation}\label{eq:Vir}
f_t(yt^{-a_i})/t^{b_i} = f^{(i)}(y) + g^{(i)}(y,t)\text{\quad and\quad} i = 1,\ldots,l,
\end{equation} where $g^{(i)}(y, t)$ collects the terms whose powers of $t$ are positive. Then $f^{(i)}(y)$ has Newton segment $I_i$ and its number of non-zero roots counted with multiplicities is $|I_i|$, the length of the interval $I_i$.

\begin{lemma}\label{L:Virotoroots}
Assume that for $i=1,\ldots,l$, the polynomial $f^{(i)}$ is a binomial. Then there exists a bijection between the set of all non-degenerate positive roots of $f_t$ for $t>0$ small enough and the set of non-degenerate positive roots of $f^{(1)},\ldots,f^{(l)}$.
\end{lemma}

\begin{proof}
Since $f^{(i)}(y)$ is a binomial, it has at most one positive root $r$ which is simple, and there will be a positive root $r_{i,t}$ of $$f^{(i)}(y) + g^{(i)}(y,t)$$ near such $r$ for $t$ small enough. Let $K\subset\ ]0,+\infty[$ denote a compact interval containing the positive root of $f^{(i)}$ for $i=1,\ldots,l$. Then, for $t>0$ small enough, the interval $K$ contains the positive root $r_{i,t}$ of $f_t(yt^{-a_i})/t^{b_i}$. Moreover, the intervals $t^{-a_1}K,\ldots,t^{-a_l}K$ are disjoint for $t>0$ small enough. This gives $l$ positive roots of $f_t$ for $t>0$ small enough. Roots of $f_t(yt^{-a_i})/t^{b_i}$ which are close to a point $r$ are positive only if $r$ is positive, and the number of these roots is determined by the first term $f^{(i)}(y)$. Since $f^{(i)}(y)$ is a binomial, it has only one simple positive root.
\end{proof} To simplify the notations, set $p_0=0$, $p_1=\lambda_1$, $p_2=-\lambda_2$, $p_3=\lambda_1 + \lambda_3,\ldots$ and $p_{n+1}=\sum_{j=0}^{n/2}\lambda_{2j+1}$ if $n$ is even and $p_{n+1}=-\sum_{j=1}^{(n+1)/2}\lambda_{2j}$ if $n$ is odd. Then by assumption, we have $p_0<p_1<\cdots<p_{n+1}$. Set $h_0=0$ and choose real numbers  $h_1,\ldots,h_{n+1}$ such that the lower part $L$ of the convex hull of $\lbrace (p_i,h_i)|\ i=0,\ldots,n+1 \rbrace$ consists of the segments $[(p_i,h_i),(p_{i+1},h_{i+1})]$ for $i=0,\ldots,n$. Therefore, projecting $L$ to $\mathbb{R}$ via the map $\mathbb{R}^2\rightarrow\mathbb{R}$ forgetting the last coordinate, we get the subdivision of $[0,p_{n+1}]$ by the intervals $[p_i,p_{i+1}]$ (see Figure ~\ref{Fig:Graph0}). Set $\alpha_1=h_1/p_1$, $\alpha_2=h_2/p_2$ and $$\alpha_i=\frac{h_i -h_{i-2}}{p_i -p_{i-2}} \text{\quad for \quad} i=3,\ldots,n+1.$$

\begin{figure}[H]
\centering
\includegraphics[scale=1.4]{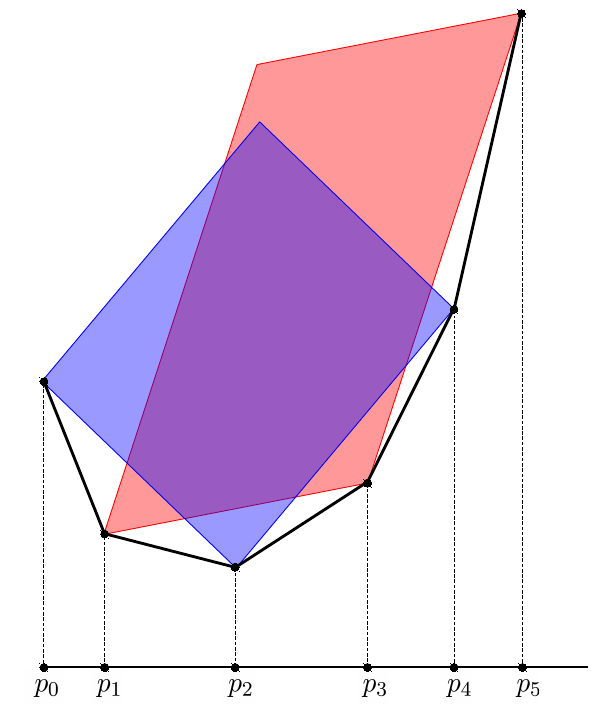}
\caption{The lower hull $L$ of $G_t$ for $n=4$.}
\label{Fig:Graph0}
\end{figure}

\begin{proposition}\label{Propn+1}
For $t>0$ small enough the polynomial ~\eqref{eq:Gale2} has $n+1$ roots in $\Delta_{+,t}=]0,+\infty[$.
\end{proposition}

\begin{proof}
It is easy to see that the lower hull of the Viro polynomial

\begin{equation}\label{eq:odd}
\prod_{j=0}^{\lfloor n/2\rfloor} P_{2j+1,t}^{\lambda_{2j+1}}(y)
\end{equation} is composed of the segments $[(p_{2j+1},h_{2j+1}),(p_{2j+3},h_{2j+3})]$ for $j=0,\ldots,\lfloor n/2\rfloor -1$. Similarly, the lower hull of 
\begin{equation}\label{eq:even}
-\prod_{j=1}^{\lfloor (n+1)/2\rfloor} P_{2j,t}^{-\lambda_{2j}}(y)
\end{equation} is composed of the segments $[(p_{2j-2},h_{2j-2}),(p_{2j},h_{2j})]$ for $j=1,\ldots,\lfloor( n+1)/2\rfloor$. It follows that the lower hull of the Viro polynomial $G_t$ is $L$. Now we apply Lemma ~\ref{L:Virotoroots} to $G_t$. For $i=0,\ldots,n$, the facial subpolynomial $G^{(i)}$ corresponding to the segment $[(p_i,h_i),(p_{i+1},h_{i+1})]\subset L$ is a binomial where one monomial comes from ~\eqref{eq:odd} and the other comes from ~\eqref{eq:even}. Consequently, this binomial has coefficients of different signs and thus it has one simple positive root. Therefore by Lemma ~\ref{L:Virotoroots}, the polynomial $G_t$ has $n+1$ non-degenerate positive roots for $t>0$ small enough.
\end{proof}
\begin{example}\label{Ex:3}
 Choose for $i=0,\ldots,n$, the slope of the segment $[(p_i,h_i),(p_{i+1},h_{i+1})]$ of $L$ to be equal to $i$. We compute explicitly the values $\alpha_1,\ldots,\alpha_{n+1}$ of the exponent of $t$ appearing respectively in $P_{1,t},\ldots,P_{n+1,t}$. We have $h_1=0$, and $$i = \frac{h_{i+1} - h_{i}}{p_{i+1} - p_i}\text{ \quad for\ } i=0,\ldots,n.$$ Since $\alpha_1=0$ and for $i=0,\ldots,n-1$, we have $\alpha_{i+2}=(h_{i+2} - h_i)/(p_{i+2} - p_i),$ then $$\alpha_{i+2}=i+\frac{p_{i+2} - p_{i+1}}{p_{i+2}-p_i}.$$ Note that $p_{i+2}-p_i=\lambda_i$ if $i$ is odd, and $p_{i+2}-p_i=-\lambda_i$ if $i$ is even. Moreover, we have $$p_{i+2} - p_{i+1}=\sum_{j=0}^{(i+1)/2}\lambda_{2j+1} + \sum_{j=1}^{(i+1)/2}\lambda_{2j} \text{\quad if $i$ is odd and\quad}-\sum_{j=1}^{(i+2)/2}\lambda_{2j} - \sum_{j=0}^{i/2}\lambda_{2j+1}\text{\quad if $i$ is even}.$$ Therefore, $$\displaystyle \alpha_{i+2}=i +\frac{\sum_{j=0}^{\lfloor i+1\rfloor/2}\lambda_{2j+1} + \sum_{j=1}^{\lfloor i+2\rfloor/2}\lambda_{2j}}{\lambda_i}.$$\qed 

\end{example}

 \providecommand{\bysame}{\leavevmode\hbox to3em{\hrulefill}\thinspace}
\providecommand{\MR}{\relax\ifhmode\unskip\space\fi MR }
\providecommand{\MRhref}[2]{%
  \href{http://www.ams.org/mathscinet-getitem?mr=#1}{#2}
}
\providecommand{\href}[2]{#2}

\end{document}